\documentclass[10pt]{amsart}
 \usepackage[margin=3cm]{geometry}
\usepackage{hyperref}
\usepackage{pb-diagram}
\usepackage{amsmath, amssymb}
\usepackage{tikz}
\usetikzlibrary{matrix,arrows,decorations.pathmorphing}
\usepackage{paralist}
\usepackage{color}

\newtheorem{thm}{Theorem}
\newtheorem{lem}[thm]{Lemma}
\newtheorem{prp}[thm]{Proposition}

\theoremstyle{definition}
\newtheorem{df}[thm]{Definition}
\newtheorem{exa}[thm]{Example}
\newtheorem{cor}[thm]{Corollary}

\theoremstyle{remark}
\newtheorem*{rem}{Remark}

\numberwithin{equation}{section}
\numberwithin{thm}{section}

\DeclareMathOperator{\map}{map}

\DeclareMathOperator{\Ch}{Ch}

\DeclareMathOperator{\Conf}{Conf}
\DeclareMathOperator{\Crit}{Crit}
\DeclareMathOperator{\Reg}{Reg}

\DeclareMathOperator{\Br}{Br}
\DeclareMathOperator{\CrSeq}{CrSeq}
\DeclareMathOperator{\CrCell}{CrCell}
\DeclareMathOperator{\dir}{dir}
\DeclareMathOperator{\Rt}{Rt}
\DeclareMathOperator{\Cell}{Cell}

\def\vP{\vec{P}}
\def\vR{\vec{\R}}

\def\R{\mathbb{R}}
\def\Z{\mathbb{Z}}

\def\cP{\mathcal{P}}

\def\cV{\mathcal{V}}
\def\cW{\mathcal{W}}
\def\cY{\mathcal{Y}}
\def\cR{\mathcal{R}}
\def\cS{\mathcal{S}}

\def\bO{\mathbf{0}}
\def\bI{\mathbf{1}}

\def\bx{\mathbf{x}}
\def\ba{\mathbf{a}}
\def\bb{\mathbf{b}}
\def\bk{\mathbf{k}}
\def\bl{\mathbf{l}}

\def\bc{\mathbf{c}}
\def\be{\mathbf{e}}

\def\vI{\vec{I}}

\def\cupdot{\mathop{\dot\cup}}

\def\Am{{A'}}

\title{Directed path spaces via discrete vector fields}
\author{Krzysztof Ziemia\'nski}
\thanks{Faculty of Mathematics, Informatics and Mechanics, University of Warsaw, Banacha 2, 02--097 Warszawa, Poland. E-mail: ziemians@mimuw.edu.pl.}

\begin{document}

\begin{abstract}
	Let $K$ be an arbitrary semi-cubical set that can be embedded in a standard cube. Using Discrete Morse Theory, we construct a CW-complex that is homotopy equivalent to the space $\vP(K)_v^w$ of directed paths between two given vertices $v,w$ of $K$. In many cases,  this construction is minimal: the cells of the constructed CW-complex are in 1--1 correspondence with the generators of the homology of $\vP(K)_v^w$. 
\end{abstract}

\maketitle

\section{Introduction}

	The spaces of directed paths on semi-cubical sets play an important role in Theoretical Computer Science \cite{FGR}, \cite{FGHMR}. In the previous paper \cite{ZPerm} the author constructed, for every bi-pointed semi-cubical set $(K,v,w)$ satisfying certain mild assumptions, a regular CW-complex $W(K)_v^w$ that is homotopy equivalent to the space of directed paths $\vP(K)_v^w$ on $K$ from $v$ to $w$. This construction is functorial, and even minimal amongst functorial constructions. The main goal of this paper is to provide a further reduction of this model.

We restrict our attention to semi-cubical sets that can be embedded into a standard cube, regarded as a semi-cubical complex. This special case is general enough to encompass most of interesting  examples appearing in Concurrency. The main result of this paper is a construction of a discrete gradient field \cite{Forman} $\cW_K$ on $W(K)_v^w$. This shows that $\cP(K)_v^w$ is homotopy equivalent to an even smaller CW-complex $X(K)$ whose cells correspond to the critical cells of $\cW_K$. Furthermore, explicit formulas describing the set of critical cells of $\cW_K$ are provided.

This construction allows to calculate the homology groups of $\cP(K)_v^w$, since the differentials in the cellular homology chain complex of $X(K)$ can be recovered using methods from \cite[Chapter 11]{Kozlov}. We do not examine these differentials in detail. It appears that in many important cases it is not necessary since the differentials vanish by dimensional reasons. This way we reprove here the result of Bjorner and Welker \cite{BW}, who calculate the homology of "not $(k+1)$--equal" configuration spaces on the real line, as well as its generalization due to Meshulam and Raussen \cite{MR}. 

We pay a special attention to the case when $K$ is a Euclidean cubical complex, i.e., a sum of cubes having integral coordinates in the directed Euclidean space $\vR^n$. Since state spaces of PV-programs \cite{D} are Euclidean cubical complexes, this case seems important for potential applications in Concurrency. Since every finite Euclidean cubical complex can be embedded into a standard cube, our results apply in this case; also, a description of critical cells of $\cW_K$ is given in this context.

\section{Preliminaries}

Let us recall some definitions and  results obtained in \cite{ZPerm}.

\emph{A d-space} \cite{Gr} is a pair $(X,\vP(X))$, where $X$ is a topological space and $\vP(X)\subseteq P(X)=\map([0,1],X)$ is a family of paths that contains all constant paths and is closed with respect to concatenation and non-decreasing reparametrizations. Paths that belong to $\vP(X)$ will be called \emph{directed paths} or \emph{d-paths}. For $x,y\in X$, $\vP(X)_x^y$ denotes the space of d-paths starting at $x$ and ending at $y$. Prominent examples of d-spaces are \emph{the directed $n$-cube $\vI^n=(I^n,\vP(\vI^n))$} and \emph{the directed Euclidean space} $\vR^n=(\R^n,\vP(\vR^n))$, where $\vP(\vI^n)$ and $\vP(\vR^n)$ are the spaces of all paths having non-decreasing coordinates.

\emph{A semi-cubical set} $K$ is a sequence of disjoint sets $(K[n])_{n\geq 0}$, equipped with \emph{face maps} $d^\varepsilon_i:K[n]\to K[n-1]$, where $n\geq 0$, $i\in\{1,\dots,n\}$ and $\varepsilon\in\{0,1\}$, that satisfy pre-cubical relations, i.e., $d_i^\varepsilon d_j^\eta = d_{j-1}^\eta d_i^\varepsilon$ for $i<j$. Elements of $K[n]$ will be called \emph{cubes} or \emph{$n$--cubes} if one needs to emphasize their dimension; $0$--cubes and $1$--cubes will be called \emph{vertices} and \emph{edges}, respectively. 
The set of all cubes of a semi-cubical set $K$ will be denoted by $\Cell(K)$ or by $K$ if it does not lead to confusion. It is partially ordered by inclusion, i.e. $c\subseteq c'$ if $c$ is the image of $c'$ under some composition of face maps.
Every cube $c\in K[n]$ has \emph{the initial vertex} $d^0(c)=d^0_1\dots d^0_1(c)$ and \emph{the final vertex} $d^1(c)=d^1_1\dots d^1_1(c)$, where $n$ face maps appear in both compositions.

\emph{The geometric realization} of a semi-cubical set $K$ is a d-space
\begin{equation}\label{e:GeometricRealization}
	|K|=\coprod_{n\geq 0} K[n]\times \vec{I}^n/(d^\varepsilon_i(c),x)\sim (c,\delta^\varepsilon_i(x)),
\end{equation}
where $\delta^\varepsilon_i(s_1,\dots,s_{n-1})=(s_1,\dots,s_{i-1},\varepsilon, s_i,\dots,s_{n-1})$. A path $\alpha\in P(|K|)$ is directed if there exist numbers $0=t_0<t_1<\dots<t_l=1$, cubes $c_i\in K[n_i]$ and directed paths $\beta_i$ in $\vec{I}^{n_i}$ such that $\alpha(t)=(c_i,\beta_i(t))$ for $t\in [t_{i-1},t_i]$.

For a semi-cubical set $K$ and a pair of its vertices $v,w\in K[0]$, \emph{a cube chain in $K$ from $v$ to $w$ in $K$} is a sequence of cubes $\bc=(c_1,\dots,c_l)$, $c_i\in K[n_i]$, $n_i>0$, that satisfies the following conditions:
\begin{itemize}
\item{$d^0(c_1)=v$,}
\item{$d^1(c_l)=w$,}
\item{$d^1(c_i)=d^0(c_{i+1})$ for $i\in \{1,\dots,l-1\}$.}
\end{itemize}
The set of all cube chains in $K$ from $v$ to $w$ is denoted by $\Ch(K)_v^w$. There is a natural partial order on $\Ch(K)_v^w$ given by the refinement of cube chains, see \cite[Definition 1.1]{ZPerm} for details.

Assume that a semi-cubical set $K$ is \emph{proper}, i.e., if $c\neq c'$ are cubes of $K$, then $\{d^0(c),d^1(c)\}\neq \{d^0(c'),d^1(c')\}$. Under this assumption, the following holds:

\begin{thm}[{\cite[Theorems 1.2 and 1.3]{ZPerm}}]
	Let $v,w\in K[0]$ be vertices of $K$.
	\begin{enumerate}[\normalfont (a)]
	\item{There is a homotopy equivalence
	\[\vP(|K|)_v^w\simeq |\Ch(K)_v^w|,\]
	where $|\Ch(K)_v^w|$ denotes the geometric realization of the nerve of $\Ch(K)_v^w$.
	}
	\item{$|\Ch(K)_v^w|$ carries a natural structure of a regular CW-complex with closed cells having the form $|\Ch_{\leq \bc}(K)|$ for $\bc\in \Ch(K)_v^w$, where $\Ch_{\leq \bc}(K)\subseteq \Ch(K)_v^w$ is a subposet of cube chains that are finer than $\bc$.\qed}
	\end{enumerate}
\end{thm}

In this paper we restrict to the case when $K$ is a semi-cubical subset of a standard cube. \emph{The standard $n$-cube $\square^n$} is a semi-cubical set whose $k$--cubes $\square^n[k]$ are sequences $(e_1,\dots,e_n)$, $e_i\in \{0,1,*\}$ having exactly $k$ entries equal to $*$. A face map $d_i^\varepsilon$ converts the $i$--th occurrence of $*$ into $\varepsilon$. It is easy to see that the geometric realization of $\square^n$ is d-homeomorphic to the directed cube $\vI^n$. Furthermore, every semi-cubical subset of $\square^n$ is proper, so the results of \cite{ZPerm} can be applied in this situation.

%For convenience, we will rather consider sub-posets of the cell poset of the standard cube than its semi-cubical subsets.

The majority of proofs in this paper is inductive with respect to the dimension of the ambient cube $\square^n$. Thus, for convenience, the coordinates will be indexed by an arbitrary finite ordered set $A$ rather than by $\{1,\dots,n\}$. In the case when $A$ is non-empty, $m\in A$ denotes its maximal element and $A'=A\setminus\{m\}$.

Let $\#X$ denote the cardinality of a finite set $X$.
\begin{df}
	\emph{The standard $A$--cube} $\square^A$ is a semi-cubical set such that
	\begin{itemize}
		\item{$\square^A[k]$ is the set of all functions $c:A\to \{0,1,*\}$ such that $\#(c^{-1}(*))=k$.}
		\item{For $c\in \square^A[k]$, if $c^{-1}(*)=\{b_1<b_2<\dots<b_k\}$, then
		\[
			d^\varepsilon_i(c)(a)=\begin{cases}
				\varepsilon & \text{if $a=b_i$}		\\		
				c(a) & \text{otherwise.}
			\end{cases}
		\] 
		}
	\end{itemize}
	\emph{An $A$--cubical complex} is a semi-cubical subset of $\square^A$.
\end{df}
We identify $|\square^A|$ with the directed $A$--cube $\vI^A$; thus, the geometric realization of an $A$--cubical complex is a subspace of $\vI^A$.

Let us introduce a notation for cubes of $\square^A$. For subsets $B_1,B_*,B_0\subseteq A$ such that $A=B_1\cupdot B_* \cupdot B_0$, let $c(B_1,B_*,B_0)$ be a cube of $\square^A$ such that
\begin{equation}
	c(B_1,B_*,B_0)(a)=\begin{cases}
		0 & \text{for $a\in B_0$}\\
		* & \text{for $a\in B_*$}\\
		1 & \text{for $a\in B_1$.}\\
	\end{cases}
\end{equation}
The dimension of $c(B_1,B_*,B_0)$ equals $\#B_*$, and $c(B_0,B_*,B_1)\subseteq c(B'_0,B'_*,B'_1)$ if and only if $B'_0\subseteq B_0$, $B'_1\subseteq B_1$ and $B_*\subseteq B'_*$.

For a subset $B\subseteq A$ and $\varepsilon\in\{0,1\}$, let $K|_B^\varepsilon\subseteq \square^B$ be the set of functions $c:B\to \{0,1,*\}$ such that the function
\begin{equation}
		A\ni a \mapsto
		\begin{cases}
			c(a) & \text{for $a\in B$}\\
			\varepsilon & \text{for $a\not\in B$}
		\end{cases}
\end{equation}
belongs to $K$. Clearly, $K|_B^\varepsilon$ is a $B$--cubical complex. After passing to geometric realizations, the restriction of $K$ corresponds to the intersection with the suitable face of the directed $A$--cube, i.e., there is a homeomorphism
\begin{equation}
	|K|^\varepsilon_B|\cong |K| \cap \{(t_a)_{a\in A}:\; \forall_{a\in A\setminus B}\; t_a=\varepsilon\}.
\end{equation}

\section{Ordered partitions and cube chains in $A$--complexes}

\begin{df}
\emph{An ordered partition} of a set $A$ is a sequence $\lambda=B_1|B_2|\dots|B_{l(\lambda)}$ of non-empty disjoint subsets of $A$ such that $A=\bigcup_{i=1}^l B_i$. We say that an ordered partition $\mu=C_1|\dots|C_{l(\mu)}$ is \emph{finer} than $\lambda$ if there exists a sequence of integers 
\[
	0=r(0)<r(1)<r(2)<\dots<r(l(\lambda))=l(\mu)
\]
such that $B_i=C_{r(i-1)+1}\cup C_{r(i-1)+2} \cup \dots\cup C_{r(i)}$ for all $i\in\{1,\dots,l(\lambda)\}$. Let $\cP_A$ be the poset of all ordered partitions of $A$, with a partial order such that $\mu \leq \lambda$ if and only if $\mu$ is finer than $\lambda$.
\end{df}

We will define an isomorphism between $\cP_A$ and the poset $\Ch(\square^A)_{\bO}^\bI$, where $\bO,\bI\in \square^A[0]$ stand for the constant functions having values $0$ and $1$ respectively.
Pick a cube chain $\bc=(c_1,\dots,c_l)\in\Ch(\square^A)_\bO^\bI$, and let $B^\bc_i=c_i^{-1}(*)$.
For every cube $c\in \square^A$ and $a\in A$ we have
\[
	d^\varepsilon(c)(a)=\begin{cases}
		c(a) & \text{for $c(a)\neq *$}\\
		\varepsilon & \text{for $c(a)=*$.}
	\end{cases}
\] 
Thus, from the condition $d^1(c_i)=d^0(c_{i+1})$ follows that
\begin{itemize}
	\item{$c_i(a)=0$ implies $c_{i-1}(a)=0$,}
	\item{$c_i(a)=1$ implies $c_{i+1}(a)=1$,}
	\item{$c_i(a)=*$ implies $c_{i-1}(a)=0$ and $c_{i+1}(a)=1$.}
\end{itemize}
Moreover,
\begin{itemize}
\item{$c_1(a)\neq 1$ for all $a\in A$, since $d^0(c_1)=\bO$,}
\item{$c_l(a)\neq 0$ for all $a\in A$, since $d^1(c_l)=\bI$.}
\end{itemize}
Thus, for every $a\in A$, the value $*$ appears exactly once in the sequence $\left(c_i(a)\right)_{i=1}^l$; all preceding elements are $0$ and all succeeding ones are $1$.
As a consequence, $\bc$ determines an ordered partition $\lambda^\bc:=B^\bc_1|B^\bc_2| \dots| B^\bc_l$ of $A$.

On the other hand, an ordered partition $\lambda=B_1|\dots| B_l$ determines a cube chain $\bc^\lambda=(c^\lambda_1,\dots,c^\lambda_l)$, where
\begin{equation}
	c^\lambda_i = c(B_1\cup\dots\cup B_{i-1},B_i,B_{i+1}\cup\dots\cup B_l).
\end{equation}

It is easy to check that these operations are mutually inverse and that finer cube chains correspond to finer partitions. As a consequence, we obtain

\begin{prp}\label{p:PChIso}
	The map
	\[
		\Ch(\square^A)_\bO^\bI\ni \bc \mapsto B_1^\bc|\dots|B^\bc_l\in \cP_A
	\]
	is an isomorphism of posets.
\end{prp}

This is a slight reformulation of \cite[Proposition 8.1]{ZPerm}.

Now let $K$ be an $A$--cubical complex. Define
\begin{equation}\label{e:PKDef}
	\cP_K:=\{\lambda\in \cP_A:\; \bc^\lambda\in \Ch(K)_\bO^\bI\}=\{\lambda\in \cP_A:\; \forall_{i\in\{1,\dots,l(\lambda)\}}\; c^\lambda_i\in K\}.
\end{equation}
This is the image of $\Ch(K)_\bO^\bI$ under the isomorphism in Proposition \ref{p:PChIso}. As a consequence, there is a sequence of homotopy equivalences
\begin{equation}
	\vP(|K|)_\bO^\bI \simeq |\Ch(K)_\bO^\bI|\simeq |\cP_K|. 
\end{equation}

The following criterion will be used later; it follows immediately from the definitions.
\begin{prp}\label{p:CompositionLemma}
	Assume that $A=C\cupdot B_1 \cupdot \dots\cupdot B_k \cupdot D$ is a partition of $A$ and $\lambda\in\cP_C$, $\mu\in \cP_D$. Then the following conditions are equivalent:
	\begin{enumerate}[\normalfont (a)]
		\item{$\lambda|B_1|\dots|B_k|\mu\in \cP_K$.}
		\item{$\lambda\in \cP(K|^0_C)$, $\mu\in \cP(K|^1_D)$ and
		\[
			c(C\cup B_1\cup\dots\cup B_{i-1},B_i,B_{i+1}\cup \dots\cup B_k\cup D)\in K
		\]
		 for every $i\in \{1,\dots,k\}$.\qed}
	\end{enumerate}	
\end{prp}

In the remaining part of the paper we will examine the poset $\cP_K$ with means of Discrete Morse Theory.

%For $\lambda=B_1|B_2|\dots|B_l\in \cP_A$ define a sequence of cubes $\bc^\lambda=(c^\lambda_k)_{k=1}^l$ such that
%\[
%	c^\lambda_k(a)=\begin{cases}
%		1 & \text{for $a\in\bigcup_{i=1}^{k-1} B_i$}\\
%		* & \text{for $a\in B_k$}\\
%		0 & \text{for $a\in\bigcup_{i=k+1}^{l} B_i$}
%	\end{cases}
%\]
%The sequence $\bc^\lambda$ is a cube chain in $K$ from $0$ to $1$ in the sense of \cite{Z1}.

\section{Discrete Morse theory for CW-posets}

In this Section we recall some basic facts from Discrete Morse Theory for regular CW-complexes. For detailed expositions of this topic see, for example, \cite{Forman} \cite{FormanUsers}, \cite[Chapter 11]{Kozlov}.

\begin{df}[\cite{Bj}]
	A poset $P$ is \emph{a CW-poset} if, for every $a\in P$, $|P_{<a}|$ is homeomorphic to a sphere $S^{d(a)-1}$. Elements of a CW-poset will be called cells and the integer $d(a)$ will be called \emph{the dimension} of a cell $a$. If $a<b\in P$ and $d(a)=d(b)-1$, then $a$ will be called \emph{a facet of $b$} and we will write $a\prec b$.
\end{df}

For a CW-poset $P$, $|P|$ has a natural CW-structure; its closed $k$--cells have the form $|P_{\leq a}|$, for $a\in P$ having dimension $k$. 
The cell poset $C(W)$ of a regular CW-complex $W$ is a CW-poset and $|C(W)|$ is homeomorphic to $W$ by a cell-preserving homeomorphism. 
In particular, the face poset of a convex polytope is a CW-poset.

\begin{df}\label{d:DVF}
	Let $P$ be a subposet of a finite CW-poset.
\begin{itemize}
	\item{ \emph{A discrete vector field $\cV$ on $P$} is a set of pairwise disjoint pairs (called \emph{vectors}) $(a,b)$, $a,b\in P$ such that $a$ is a facet of $b$.}
	\item{\emph{A flow} of $\cV$ is a sequence of cells of $P$
	\begin{equation}
		(a_1,b_1,a_2,b_2,\dots,a_{k},b_k,a_{k+1})\label{e:Flow}
	\end{equation}
	such that, for all $i\in\{1,\dots,k\}$,
\begin{itemize}
	\item{$(a_i,b_i)\in\cV$,}
	\item{$(a_{i+1},b_i)\not\in\cV$}
	\item{$a_{i+1}\prec b_i$.}
\end{itemize}
	}
	\item{\emph{A cycle} of $\cV$ is a flow (\ref{e:Flow}) such that $a_1=a_{k+1}$, $k>0$.}
	\item{A discrete vector field $\cV$ is \emph{a gradient field} if it admits no cycle.}
\end{itemize}	
\end{df}	 

For a discrete vector field $\cV$ on a CW-poset $P$, let
\begin{equation}
	\Reg(\cV)=\bigcup_{(a,b)\in\cV} \{a,b\}
\end{equation}
be the set of \emph{regular cells} of $\cV$, and let $\Crit(\cV)=P\setminus\Reg(\cV)$ be the set of \emph{critical cells} of $\cV$. 
If $P\subseteq Q$ and $Q$ is a CW-poset, then $\cV$ can be regarded as a discrete vector field on $Q$. In such case, we will write $\Crit_{P}(\cV)$ or $\Crit_Q(\cV)$ for the set of critical cells to emphasize which underlying poset we have in mind.

The importance of gradient fields follows from the following theorem:

\begin{thm}[{\cite[Theorem 11.13]{Kozlov}}]\label{t:Forman}
	Assume that $P$ is a CW-poset and $\cV$ is a gradient field on $P$. Then there exists a CW-complex $W(\cV)$ that is homotopy equivalent to $|P|$, whose $d$--dimensional cells are in 1-1 correspondence with $d$--dimensional critical cells of $\cV$.
\end{thm}

For convenience, we will use a notion of discrete Morse function which is slightly different from the original one.

\begin{df}
	Let $\cV$ be a discrete vector field on a finite CW-poset $P$. \emph{A (discrete) Morse function} associated to $\cV$ is a function $h:P\to H$, where $H$ is an ordered set, such that, for every $a\prec b\in P$,
	\begin{enumerate}[(a)]
	\item{$(a,b)\in\cV$ implies that $h(a)>h(b)$,}
	\item{$(a,b)\not\in \cV$ implies that $h(a)\leq h(b)$.}
	\end{enumerate}
\end{df}

\begin{lem}\label{l:Order}
	Let $\cV$ be a discrete vector field on a finite CW-poset $P$. If there exists a Morse function associated to $\cV$, then $\cV$ is a gradient field.
\end{lem}
\begin{proof}
	If $(a_1,b_1,\dots,a_k,b_{k+1})$ is a flow in $\cV$, then
	\[
		h(a_1)>h(b_1)\geq h(a_2)>\dots\geq h(a_k)>h(a_{k+1});
	\]
	thus, $a_1\neq a_{k+1}$.
\end{proof}

We say that a subposet $Q$ of a poset $P$ is \emph{closed} if, for $x\leq y\in P$, $y\in Q$ implies that $x\in Q$. 

\begin{lem}\label{l:Div}
	Let $P$ be finite CW-poset and let $Q\subseteq P$ is a closed subposet. Let $\cV_Q, \cV_{P\setminus Q}$ be a discrete vector fields on $Q$ and $P\setminus Q$ respectively. Then every cycle of $\cV=\cV_Q\cup \cV_{P\setminus Q}$ is contained either in $Q$ or in $P\setminus Q$. In particular, if both $\cV_Q$ and $\cV_{P\setminus Q}$ are gradient fields, then $\cV$ is also a gradient field.
\end{lem}
\begin{proof}
	Let $(a_1,b_1,\dots,a_k,b_k,a_1)$ be a cycle. By the assumptions, $a_i\in Q$ implies $b_i\in Q$, and $b_i\in Q$ implies $a_{i+1}\in Q$, since $a_{i+1}\prec b_i$. Thus, either all elements of the cycle are in $Q$ or none is.
\end{proof}

We will give now some examples of gradient fields. While the first two examples  are not crucial in proving main results of this paper, they can be helpful in understanding similar constructions performed on permutahedra.

In all following examples, $A$ is a non-empty finite ordered set, $m$ is its maximal element and $A'=A\setminus \{m\}$.

\begin{exa}[Simplices]
	\emph{The $A$--simplex} is a poset $\Delta^A$ of non-empty subsets of $A$. The standard vector field on $\Delta^A$ is
	\[
		\cS^\Delta_A=\{(B,B\cup\{m\}):\; \emptyset\neq B\subseteq A'\}.
	\]
	This is a gradient field, since
	\[
		h:\Delta^A\ni B \mapsto \begin{cases}
			0 & \text{if $m\in B$}\\
			1 & \text{if $m \not\in B$}
		\end{cases}
		\in \{0,1\}
	\]
	is a Morse function. 	The only critical cell of $\cS^\Delta_A$ is $\{m\}$.
\end{exa}

\begin{exa}[Cubes]
	Here $\square^A=\Cell(\square^A)$ denotes the poset of cubes of the standard $A$--cube.	
	Define $\square^{A,m}=\{c\in \square^A:\; c(m)=0\}$ and $\square^{A,r}=\square^A\setminus \square^{A,m}$. Clearly, $\square^{A,m}$ is a closed subposet of $\square^A$, which is isomorphic  to $\square^{A'}$.
	We define vector fields $\cS^{\square m}_A$, $\cS^{\square r}_A$ and $\cS^{\square}_A$ on posets $\square^{A,m}$, $\square^{A,r}$ and $\square^A$, respectively, inductively as follows. We put
	\[\cS^{\square m}_\emptyset=\cS^{\square r}_\emptyset=\cS^{\square}_\emptyset=\emptyset,\]
	and for $A\neq\emptyset$ let
	\begin{enumerate}[\normalfont (a)]
	\item{$\cS^{\square m}_A=\{(f,g):\; (f|_\Am,g|_{\Am})\in \cS^\square_{\Am}, \; f(m)=g(m)=0\}.$}
	\item{$\cS^{\square r}_A=\{(f,g):\; f|_{\Am}=g|_{\Am},\; f(m)=1,\; g(m)=*\}$}
	\item{$\cS^{\square}_A=\cS^{\square m}_A\cup \cS^{\square r}_A$.}
	\end{enumerate}
	We will prove inductively that $\cS^\square_A$ is a gradient field. This is obvious for $A=\emptyset$. By the inductive hypothesis, $\cS^{\square m}_A\simeq \cS^\square_{A'}$ is a gradient field on $\square^{A,m}\simeq \square^{A'}$, and $\cS^{\square r}_{A}$ admits a Morse function
	\[
		\square^{A,r}\ni c \mapsto c(m) \in \{*<1\}
	\]
	and hence it is also a gradient field. Now Lemma \ref{l:Div} implies that $\cS^A_\square$ is a gradient field. The only critical cell of $\cS^\square_A$ is a $0$--cell $\bO=(0,\dots,0)$.
\end{exa}

\begin{exa}[Product of discrete vector fields]
	Let $P,Q$ be finite CW-posets, and let $\cV$, $\cW$ be discrete vector fields on $P$ and $Q$, respectively. Define a discrete vector field $\cV\times \cW$ on $P\times Q$ by
	\begin{equation}
		\cV\times \cW=\{((p,q),(p,q')):\; p\in P,\; (q,q')\in \cW\}
		\cup \{((p,q),(p',q)):\; (p,p')\in\cV,\; q\in\Crit(\cW)\}.
	\end{equation}
	We have $\Crit(\cV\times \cW)=\Crit(\cV)\times \Crit(\cW)$. Notice that gradient fields $\cV\times \cW$ and $\cW\times \cV$ are not, in general, equal.
\end{exa}

\begin{prp}
	Assume that $P,Q$ are finite CW-posets and $\cV$ and $\cW$ are gradient fields in $P$ and $Q$, respectively. Then $\cV\times \cW$ is a gradient field on $P\times Q$.
\end{prp}
\begin{proof}
Let 
	\[
		((p_1,q_1),(p_2,q_2),\dots,(p_{2k+1},q_{2k+1})=(p_1,q_1))
	\]	
	be a cycle of $\cV\times \cW$, i.e., $((p_{2i-1},q_{2i-1}),(p_{2i},q_{2i}))\in \cV\times \cW$, $(p_{2i+1},q_{2i+1})\prec (p_{2i},q_{2i})$. Assume that $p_i$ are not all equal, and let 
	\[p_{i(1)},p_{i(2)},\dots,p_{i(r)}=p_{i(1)},\qquad i(1)<i(2)<\dots<i(r)\]
	 be all different values of $p_i$. For $s\in\{1,\dots,r-1\}$, the dimensions of $p_{i(s)}$ and $p_{i(s+1)})$ differ by $1$; if $\dim(p_{i(s)})=\dim(p_{i(s+1)})-1$, then $(p_{i(s)},p_{i(s+1)})\in \cV$ which implies that $(p_{i(s+1)},p_{i(s+2)})$, since no cell may belong to two different vectors, and then $p_{i(s+2)}\prec p_{i(s+1)}$. As a consequence, either $(p_{i(1)},p_{i(2)},\dots,p_{i(r)}=p_{i(1)})$ or $(p_{i(2)},p_{i(3)},\dots,p_{i(r)}=p_{i(1)},p_{i(2)})$ is a cycle of $\cV$. If all $p_i$'s are equal, then $(q_1,\dots,q_{2k+1})$ is a cycle of $\cW$. In both cases we get a contradiction.
\end{proof}

Notice that the standard gradient fields on cubes can be defined alternatively by formulas $\cS^\square_{\{x\}}=\{(1,*)\}$,  $\cS^\square_A=\cS^\square_{\{m\}}\times \cS^\square_\Am$.

%The condition (b) is equivalent to the following: if $\preceq$ is a transitive and reflexive relation spanned by $a\preceq b$ for $(a,b)\in\cV$ and $b\preceq a$ for $a\in d(b)$, $(a,b)\not\in \cV$, then $\preceq$ is a partial order.

\section{Permutahedra}

The main goal of this Section is to construct "standard" gradient fields on permutahedra. As before, $A$ is a finite ordered set, $m\in A$ is a  maximal element and $A'=A\setminus \{m\}$. We will write $A=B_1\cupdot\dots\cupdot B_n$ when $B_1,\dots,B_n$ are pairwise disjoint family of subsets of $A$ such that $\bigcup B_i=A$ and the order on every $B_i$ is inherited from $A$.

Recall that $\cP_A$ denotes the poset of ordered partitions of $A$. $\cP_A$ is a CW-poset whose geometrical realization is a permutahedron on letters $A$ \cite[p. 18]{Ziegler}. For a cell $\lambda=B_1|B_2|\dots|B_l\in\cP_A$, we have
\[
	\dim(\lambda)=(\# B_1-1)+(\# B_2-1)+\dots+(\# B_l-1).
\]

%We will define a "standard" gradient $\cV_A$ field on $\cP_A$. The idea of the construction is inductive. There is a natural embedding 
%\[	\cV_{A'}\ni \lambda \mapsto \lambda|m \in \cV_A\]
%which allows to  : we start with a gradient field $\cV_{A'}$

Denote 
\begin{equation}
	\cP_A^m := \cP_{A'}|m = \{\lambda|m:\; \lambda\in\cP_{A'}\},\quad \cP^r_A=\cP^A\setminus \cP_A^m.
\end{equation}
Clearly, $\cP_A^m$ is a closed subposet of $\cP_A$, which is isomorphic to $\cP_{A'}$. 

Define discrete vector fields $\cV^{r}_{A}$ on $\cP_A^r$, $\cV^{m}_{A}$ on $\cP_A^m$ and $\cV_{A}=\cV^{r}_{A}\cup \cV^{m}_{A}$ on $\cP_A$ inductively in the following way. We put $\cV^{r}_\emptyset=\cV^{m}_\emptyset=\cV_\emptyset=\emptyset$ 	and, for a $A\neq\emptyset$,
	\begin{align}
		\cV^m_A&=\cV_{A'}|m=\{(\pi|m, \varrho|m):\; (\pi,\varrho)\in \cV_{A'}\}\\
		\cV^r_A & =\{	(\pi|m|B|\varrho, \pi|m\cup B|\varrho):\; \pi\in \cP_C,\; \varrho\in \cP_D,\; B\neq\emptyset,\; A'=C\cupdot B\cupdot D  \}.
	\end{align}

\begin{prp}
	The only critical cell of $\cV_A$ is 
\[
	u_A=a_1|a_2|\dots|a_n,\qquad \{a_1<a_2<\dots<a_n\}=A
\]
having dimension $0$.
\end{prp}
\begin{proof}
	Immediately from the definition follows that $\Reg(\cV_A^r)=\cP_A^r$ and
	\[
		\Crit_{\cP^m_A}(\cV^m_A)=\Crit_{\cP_{A'}}(\cV_{A'})|m=\{u_{A'}|m\}=\{u_A\}.
	\]	
	Thus, $\Crit_{\cP_A}(\cV_A)=\Crit_{\cP^m_A}(\cV^m_A)\cup \Crit_{\cP^r_A}(\cV^r_A)=u_A$.
\end{proof}

\begin{prp}
	$\cV_A$ is a gradient field.
\end{prp}
\begin{proof}
	This is obvious for $A=\emptyset$, so assume otherwise. There is an isomorphism $(\cP^m_A,\cV^m_A)\cong (\cP_{A'},\cV_{A'})$;  hence, by the inductive hypothesis, $\cV^m_A$ is a gradient field on $\cP^m_A$. By Lemma \ref{l:Div}, it remains to prove that $\cV^r_A$ is a gradient field on $\cP^r_A$.
	
	Define an order $\dot\leq$ on the set $H=\Z\times \Z_+$ in the following way: $(s,t)\mathrel{\dot\leq} (s',t')$ if and only if one of the following conditions is satisfied:
	\begin{itemize}
		\item{$s>s'$,}
		\item{$s=s'$ and $1\neq t\leq t'$,}	
		\item{$s=s'$ and $t'=1$.}	
	\end{itemize}
	This is the lexicographic order on the product, with the inverse order on $\Z$ and the order
	\[
		2<3<4<\dots<1
	\]
	on $\Z_+$. Every element $\lambda\in\cP^r_A$ can be written uniquely as
	\begin{equation}\label{e:LambdaPresentation}
		\lambda=\pi|B|\varrho,
	\end{equation}
	where  $m\in B$, $\pi\in\cP(C)$, $\varrho\in\cP(D)$, $A=C\cupdot B\cupdot D$. Let
	\begin{equation}\label{e:hA}
		h_A(\lambda)=(\# C, \#B).
	\end{equation}
	We will prove that $h_A:\cP^r_A\to H$ is a Morse function. Assume that $\mu\prec\lambda$ and $\lambda$ has a presentation (\ref{e:LambdaPresentation}). If $\mu=\pi'|B|\varrho$ for $\pi'\prec\pi\in\cP_C$, or $\mu=\pi|B|\varrho'$ for $\varrho'\prec\varrho\in \cP_D$, then $(\mu,\lambda)\not\in\cV^r_A$ and $h_A(\lambda)=h_A(\mu)$. Assume otherwise, i.e., that $\mu=\pi|B_1|B_2|\varrho$ for $B_1\cupdot B_2=B$, $B_1,B_2\neq\emptyset$. Consider the following cases:
	
	\begin{itemize}
	\item{$B_1=\{m\}$. Then $(\mu,\lambda)\in\cV^r_A$, and
	\[
		h_A(\mu)=(\#C,1)\mathrel{\dot>}(\#C,\#B)=h_A(\lambda).
	\]
	}
	\item{$\{m\}\subsetneq B_1$. Then $(\mu,\lambda)\not\in\cV^r_A$ and 
	\[
		h_A(\mu)=(\#C,\#B_1)\mathrel{\dot<}(\#C,\#B)=h_A(\lambda),
	\]
	since $\#B_1>1$.
	}
	\item{$m\in B_2$. Then $(\mu,\lambda)\not\in\cV^r_A$ and 
	\[
		h_A(\mu)=(\#C+\#B_1,\#B_2)\mathrel{\dot<}(\#C,\#B)=h_A(\lambda).
	\]
	}
	\end{itemize}
	This proves that $\cV^r_A$ is a gradient field.
\end{proof}

The following picture illustrates the gradient field $\cV_{A}$ for $A=\{1,2,3\}$.
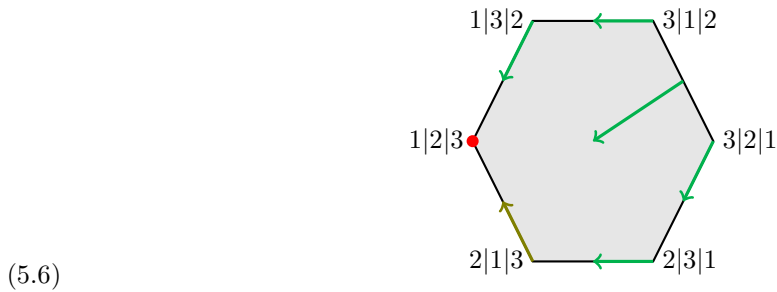
\begin{equation}
\begin{tikzpicture}[scale=0.8]
	\fill[color=gray!20!white] 	(0,0) -- (1,2)--(3,2)--(4,0)--(3,-2)--(1,-2)--(0,0);
	\draw[thick] (0,0) -- (1,2)--(3,2)--(4,0)--(3,-2)--(1,-2)--(0,0);
	\node[left] at (0,0) {$1|2|3$};
	\node[left] at (1,2) {$1|3|2$};
	\node[left] at (1,-2) {$2|1|3$};
	\node[right] at (3,2) {$3|1|2$};
	\node[right] at (4,0) {$3|2|1$};
	\node[right] at (3,-2) {$2|3|1$};
	\draw[very thick,->,color=blue!30!green] (1,2)--(0.5,1);
	\draw[very thick,->,color=blue!30!green] (3,2)--(2,2);
	\draw[very thick,->,color=blue!30!green] (4,0)--(3.5,-1);
	\draw[very thick,->,color=blue!30!green] (3,-2)--(2,-2);
	\draw[very thick,->,color=blue!30!green] (3.5,1)--(2,0);
	\draw[very thick,->,color=red!50!green] (1,-2)--(0.5,-1);
	\fill[color=red] (0,0) circle[radius=0.1];
\end{tikzpicture}
\end{equation}

\section{A gradient field on $\cP_K$}

In this Section, we construct a gradient field on $\cP_K$, for any finite ordered set $A$ and an $A$--cubical complex $K\subseteq \square^A$. The starting point is the restriction of $\cV_K$ to $\cP_A$, i.e.,
\[
	\cV_K=\cV_A|_{\cP_K}=\{(\lambda,\mu)\in \cV_A:\; \lambda,\mu\in \cP_K\}.
\]
%Similarly, define $\cV^r_K=\cV^r_A|_{\cP_K}$, $\cV^m_K=\cV^m_A|_{\cP_K}$.
In general, the discrete vector field $\cV_K$ has critical cells that are not critical cells of $\cV_A$: if $(\lambda,\mu)\in \cV_A$, $\lambda\in \cP_K$ and $\mu\not\in \cP_K$, then $\lambda\in\Crit(\cV_K)$. We will add some vectors to $\cV_K$ to reduce the number of critical cells.

Denote $\cP^r_K=\cP_A^r\cap \cP_K$, $\cP^m_K=\cP_A^m\cap \cP_K$. Note that $\cP^m_K$ is a closed subposet of $\cP_K$, which is empty if $c(A',\emptyset,m)\not\in K$ and otherwise there is an isomorphism of CW-posets
\begin{equation}
	\cP_{K|^0_{A'}}\ni \lambda \mapsto \lambda|m \in \cP_K.
\end{equation}

\begin{df}
	A cube $c(C,B\cup\{m\},D)\in \square^A$ is \emph{a branching cube} of $K$ if
\begin{itemize}
	\item{$c(C,m\cup B,D)\not\in K$,}
	\item{$c(C,m,B\cup D),c(C\cup m, B,D)\in K$.}
\end{itemize}
	These conditions imply that $B\neq \emptyset$.
	Sequences $(C,B,D)$ such that $c(C,B\cup\{m\},D)$ is a branching cube of $K$ will be called \emph{branching sequences} of $K$. Let $\Br(K)$ be the set of all branching sequences of $K$.
\end{df}

 For $(C,B,D)\in \Br(K)$ let
\begin{equation}
	\cR_{(C,B,D)}=\{\pi|m|B|\varrho:\; \pi\in \cP_{K|_C^0},\; \varrho\in \cP_{K|_D^1}\}\subseteq \cP^r_K
\end{equation}
and let 
\begin{equation}
	\cR_K=\bigcup_{(C,B,D)\in \Br(K)} \cR_{(C,B,D)} \subseteq \cP^r_K.
\end{equation}
Clearly, the posets $\cR_{(C,B,D)}$ are pairwise disjoint, and there is an isomorphism of posets
\begin{equation}
	\cR_{(C,B,D)}\cong \cP_{K|_C^0}\times \cP_{K|_D^1},
\end{equation} 
which shifts the dimensions of elements by $\#B-1$.

%Geometrically, $K|_{B}^\varepsilon$ is the intersection of $K$ with a hyperplane $\bigcap_{b\in B} \{b=\varepsilon\}$.

\

\begin{prp}
	For an $A$--cubical complex $K$, we have $\Crit_{\cP^r_K}(\cV^r_K)=\cR_K$.
\end{prp}
\begin{proof}
	If $\lambda=\pi|m|B|\varrho\in \cR_{(C,B,D)}$, then $(\pi|m|B|\varrho,\pi|m\cup B |\varrho)\in \cV^r_A$ and $c(C,m\cup B,D)\not\in K$; as a consequence, $\pi|m\cup B |\varrho\not\in \cP^r_K$ and then $\lambda\in \Crit_{\cP^r_K}(\cV^r_K)$. This proves that $\cR_K\subseteq \Crit_{\cP^r_K}(\cV^r_K)$.

	Assume that $\lambda\in \Crit_{\cP^r_K}(\cV^r_K)$. Since $\Crit_{\cP^r_A}(\cV^r_A)=\emptyset$, there exists $\mu\in \cP_A^r$ such that $(\lambda,\mu)\in \cV^r_A$ and $\mu\not\in\cP^r_K$. The definition of $\cV^r_A$ implies that
\[
	\lambda=\pi|m|B|\varrho,\qquad \mu=\pi|m\cup B|\varrho
\]	
	for $A=C\cupdot m \cupdot B\cupdot D$, $\pi\in \cP_C$, $\varrho\in \cP_D$. By \ref{p:CompositionLemma}, $\lambda\in \cP_K$ implies that $\pi\in\cP_{K|^0_C}$, $\varrho\in \cP_{K|^1_D}$ and $c(C,m,B\cup D),c(C\cup m, B,D)\in K$. Thus, since $\mu\not\in \cP_{K}$, $c(C,m\cup B,D)$ cannot belong to $K$. As a consequence, $(C,B,D)\in \Br(K)$ and then $\lambda\in\cR_{(C,B,D)}\subseteq \cR_K$.
\end{proof}

As a consequence, there is a decomposition 
\begin{equation}\label{e:Decomposition}
	\cP_K=\cP_K^m \cupdot  \Reg(\cV^r_K)\cupdot \cR_K.
\end{equation}
We will define inductively discrete vector fields on the components of this decomposition; they are empty if $A=\emptyset$, and otherwise they are inductively defined by the following formulas:
\begin{itemize}
	\item{On $\cP_K^m$:
	\[
		\cW^m_K=\cW_{K|^0_{A'}}|m=\{(\lambda|m,\mu|m):\; (\lambda,\mu)\in \cW_{K|^0_{A'}}\}
	\]
	if $c(\emptyset, A', m)\in K$; otherwise, $\cW^m_K=\emptyset$. Notice that if $\lambda|m\in \cP^m_K$ and $\mu\in \cP_{K|^0_{A'}}$, then also $\mu|m\in \cP^m_K$, which guarantees that this definition is valid.
	}
	\item{On $\Reg(\cV^r_K)$ we take $\cV^r_K$.}
	\item{For $(C,B,D)\in \Br(K)$, let $\cY_{(C,B,D)}$ be a discrete vector field on $\cR_{(C,B,D)}$
	\begin{multline*}
		\cY_{(C,B,D)}=\cW_{K|_{C}^0}|m|B|\cW_{K|_D^1}=\\
		\{(\pi|m|B|\varrho, \pi|m|B|\varrho'):\; \pi\in\cP_{K|_C^0}, (\varrho,\varrho')\in \cW_{K|_D^1}\}\\
		\cup \{(\pi|m|B|\varrho, \pi'|m|B|\varrho):\; (\pi,\pi')\in\cW_{K|_C^0}, \varrho\in \Crit(\cW_{K|_D^1})\}.
	\end{multline*}
	This is isomorphic to the product discrete vector field $\cW_{K|_{C}^0}\times \cW_{K|_D^1}$ on $\cP_{K|_{C}^0}\times \cP_{K|_{D}^1}$ via the isomorphism
	\[
		\cP_{K|_{C}^0} \times \cP_{K|_{D}^1}\ni (\lambda,\mu) \mapsto \lambda|m|B|\mu \in \cR_{(C,B,D)}
	\]
	of the underlying posets.
	}
	\item{On $\cR_K$:
	\[
		\cY_K=\bigcup \cY_{(C,B,D)}.
	\]
	}
\end{itemize}
Finally, we define $\cW^r_K=\cV^r_K \cup \cY_K$ and
\begin{equation}
	\cW_K=\cW^m_K\cup \cW^r_K=\cW^m_K  \cup \cV^r_K \cup \cY_K.
\end{equation}
An easy inductive argument shows that $\cV_K\subseteq \cW_K$.

Recall (\ref{e:hA}) that $h_A:\cP^r_A\to H$ is a weak Morse function associated to $\cV^r_A$.

\begin{prp}\label{p:WeakInc}
	If $(\lambda,\mu)\in \cY^r_K$, then $h_A(\lambda)=h_A(\mu)$. As a consequence, $(\lambda,\mu)\in \cW^r_K$ implies that $h_A(\lambda)\mathrel{\dot\geq} h_A(\mu)$.
\end{prp}
\begin{proof}
	We have $h_A(\lambda)=h_A(\mu)=(\#C,1)$ for $(\lambda,\mu)\in \cY_{(C,B,D)}$.
\end{proof}

\begin{prp}\label{p:WKIsGradient}
	$\cW_K$ is a gradient field.
\end{prp}
\begin{proof}
	Proof by induction with respect to the cardinality of $A$. By the inductive hypothesis, $\cW^m_K\simeq \cW_{K|^0_{A'}}$ is a gradient field on $\cP^m_K\simeq \cP_{K|^0_{A'}}$. Since $\cP^m_K$ is closed in $\cP_K$, by \ref{l:Div} it remains to prove that $\cW^r_K$ is a gradient field on $\cP^r_K$. Assume that  
	\[
		(\lambda_1, \mu_1, \lambda_2, \mu_2, \dots, \mu_k,\lambda_1)
	\]
	is a cycle of $\cW^r_K$. By \ref{p:WeakInc}, all values $h_A(\lambda_i)$, $h_A(\mu_i)$ are equal; thus, $(\lambda_i,\mu_i)\in \cY_K$ for all $i$, since $h_A$ is a weak Morse function of $\cV^r_K$.
	For every $i$, $\lambda_i$ and $\mu_i$ must lie in the same component $\cR_{(C_i,B_i,D_i)}$. Moreover, $C_{i+1}=C_i$ and $B_{i+1}\subseteq B_i$; this implies that the triples  $(C_i,B_i,D_i)$ are equal for all $i$. Thus, this cycle is a cycle of $\cY_{(C,B,D)}$ for some $(C,B,D)\in \Br(K)$; this leads to a contradiction since $\cY_{C,B,D}$ is isomorphic to $\cW_{K|^0_C}\times \cW_{K|^1_D}$, which is a gradient field by the inductive hypothesis.
\end{proof}

As a consequence of \cite[Theorem 1.2]{ZPerm} and Theorem \ref{t:Forman} we obtain
\begin{cor}
	For an $A$--cubical complex $K\subseteq \square^A$, the space $\vec{P}(K)_{\bO}^\bI$ is homotopy equivalent to a CW-complex whose $k$--cells correspond to $k$--dimensional critical cells of $\cW_K$.
\end{cor}

The following inductive formula for critical cells of $\cW_K$ is an immediate consequence of the definition of $\cW_K$:
\begin{prp}\label{p:CritDesc}
	Let $K$ be an $A$--cubical complex. If $A=\emptyset$, then $\Crit_{\cP_K}(\cW_K)=\cP_K$; if $A\neq\emptyset$
	\begin{align*}
		\Crit_{\cP^m_K}(\cW^m_K)&=\begin{cases}
			\{\lambda|m\in \cP_K:\; \lambda\in \Crit(\cW_{K|^0_{A'}})\} & \text{if $c(\emptyset,A',m)\in K$}\\
			\emptyset & \text{otherwise}
		\end{cases}		\\
		\Crit_{\cP^r_K}(\cW^r_K)&=\bigcup_{(C,B,D)\in \Br(K)} \{\pi|m|B|\varrho:\; \pi\in \Crit(\cW_{K|^0_C}),\; \varrho\in\Crit(\cW_{K|^1_D})\},\\
		\Crit_{\cP_K}(\cW_K)&=\Crit_{\cP^m_K}(\cW^m_K)\cup \Crit_{\cP^r_K}(\cW^r_K)
	\end{align*}
\end{prp}
In the next Section we will obtain an explicit formula for the critical cells of $\cW_K$.

\section{Explicit formula for the critical cells}

For any finite ordered set $B=\{b_1<b_2<\dots<b_l\}$ define $\tau_B,\kappa_B\in \cP_B$ by
\begin{equation}\label{e:TauKappa}
	\tau_B=b_1|b_2|\dots|b_l,\qquad \kappa_B=b_l | \{b_1,b_2,\dots,b_{l-1}\};
\end{equation}
for $\kappa_B$ we require that $B$ has at least two elements.

\begin{df}\label{d:CrSeq}
	\emph{A critical sequence} in an $A$--cubical complex $K$ is a pair of sequences of subsets $((E_j)_{j=1}^q,(F_j)_{j=0}^q)$ of $A$ such that
\begin{enumerate}[(a)]
	\item{$A=E_1\cupdot \dots\cupdot E_q\cupdot F_0\cupdot F_1\cupdot \dots \cupdot F_q$.}
	\item{\emph{The critical cell}
		\[
			\sigma((E_j),(F_j)):=\tau_{F_0}|\kappa_{E_1}|\tau_{F_1}|\kappa_{E_2}|\dots|\tau_{F_{q-1}}|\kappa_{E_q}|\tau_{F_q}\in \cP_A
		\]
		associated to $((E_j),(F_j))$ belongs to $\cP_K$.}
	\item{For every $j\in\{1,\dots,q\}$, either $F_{j-1}=\emptyset$ or $\max(F_{j-1})<\max(E_j)$,}
	\item{For every $j\in\{1,\dots,q\}$, $c(C_j,E_j,D_j)\not\in K$, where
	\begin{align}\label{e:CDSetsDef}
		C_j&= F_0\cup E_1\cup F_1 \cup E_2\cup \dots\cup E_{j-1}\cup F_{j-1}\\ 
		D_j&= F_j\cup E_{j+1}\cup F_{j+1} \cup E_{j+2}\cup \dots\cup E_{q}\cup F_{q} \notag
	\end{align}
	}
	If (b) is satisfied, this is equivalent to the condition
	$\tau_{F_0}|\kappa_{E_1}|\dots |\tau_{F_{j-1}}|E_j|\tau_{F_j}|\dots|\kappa_{E_q}|\tau_{F_q}\not\in \cP_K$. Let $\CrSeq(K)$ be the set of all critical sequences in $K$.
	
\end{enumerate}
	We do not require that the sets $F_j$ are non-empty but the conditions (b) and (d) imply that every $E_j$ has at least two elements.
	\emph{The dimension} of a critical sequence $((E_j)_{j=1}^q,(F_j)_{j=0}^q)$ is
	\begin{equation}
		\dim((E_j)_{j=1}^q,(F_j)_{j=0}^q) = \sum_{j=1}^q (\# E_j - 2),	
	\end{equation}
	which is equal to the dimension of the associated critical cell.
	Let $\CrSeq^d(K)$ denote the set $d$--dimensional critical sequences.
	
	Let $\CrCell(K)\subseteq \cP_K$ be the set of critical cells in $K$, i.e., the cells that are associated to a critical sequence, and let $\CrCell^d(K)\subseteq \CrCell(K)$ be the subset of $d$--dimensional cells. Notice that there are bijections
\[
	\CrSeq(K)\cong \CrCell(K)\quad \text{and} \quad \CrSeq^d(K)\cong \CrCell^d(K)
\]	
	since every critical cell $\lambda$ determines a unique critical sequence $((E_j),(F_j))\in \CrSeq(K)$ such that $\lambda=\sigma((E_j),(F_j))$, and the dimensions of critical sequences and of the associated critical cells coincide.
\end{df}

\begin{prp}\label{p:ExplicitDescOfCrit}
	For every $A$--cubical complex $K$ and $d\geq 0$, we have
	\[
		\Crit^d(\cW_K)=\CrCell^d(K)\cong \CrSeq^d(K).
	\]
\end{prp}

\begin{proof}
	This is obvious if $A=\emptyset$, so we assume otherwise and proceed inductively. Assume that $((E_j)_{j=1}^q,(F_j)_{j=0}^q)\in\CrSeq(K)$; we will show that $\lambda:= \sigma((E_j),(F_j))\in\Crit(\cW_K)$. There are two cases to consider:
	\begin{itemize}
	\item{
	\textit{There exists $r$ such that $m\in E_r$.} Then
\begin{align*}
	((E_j)_{j=1}^{r-1},(F_j)_{j=0}^{r-1}) &\in\CrSeq(K|^0_{C_j})\\
	((E_j)_{j=r+1}^q,(F_j)_{j=r}^q) &\in \CrSeq(K|^1_{D_j})
\end{align*}	
	where $C_j, D_j$ are defined as in (\ref{e:CDSetsDef}). From the inductive hypothesis,
	\begin{align*}
		\pi & := \sigma((E_j)_{j=1}^{r-1},(F_j)_{j=0}^{r-1})\in\Crit(K|^0_{C_j}) \text{ and} \\	
		\varrho &:= \sigma((E_j)_{j=r+1}^{q},(F_j)_{j=r}^{q})\in\Crit(K|^1_{D_j}).
	\end{align*}
	The conditions (b) and (d) imply that $(C_j,E_j\setminus\{m\}, D_j)\in\Br(K)$. Therefore, from \ref{p:CritDesc} follows that
\[
	\lambda=\sigma((E_j)_{j=1}^q,(F_j)_{j=0}^q)=
	\pi | \kappa_{E_j} |\varrho=
	\pi | m | E_j\setminus\{m\}| \varrho
	  \in \Crit(\cW_K).
\]
	}
	\item{
	\textit{$m\not\in B_r$ for every $r$.} Then the condition (c) guarantees that $m\in F_q$. Let
	\[
		\lambda'=\sigma((E_j)_{j=1}^q, (F_0,F_1,\dots, F_{q-1}, F_q\setminus\{m\}));
	\]
	it is easy to check that $\lambda'\in \CrCell(K|^0_{A'})$. As above, from the inductive hypothesis and \ref{p:CritDesc} follows that $\lambda=\lambda'|m\in\Crit(\cW_K)$.
	}
	\end{itemize}
	
	Now assume that $\lambda\in \Crit(\cW_K)$. Again, by \ref{p:CritDesc} there are two cases:
	\begin{itemize}
	\item{
		\textit{$\lambda=\lambda'|m$ for $\lambda'\in\Crit(\cW_{K|^0_{A'}})$}. By the inductive hypothesis,
		 $\lambda'\in\CrCell(K|^0_{A'})$ and then $\lambda'=\sigma((E_j)_{j=1}^q,(F_j)_{j=0}^q)$ for $((E_j),(F_j))\in\CrSeq(K)$. Clearly,
	\[
		((E_j)_{j=1}^q, (F_0,\dots,F_{q-1},F_{q}\cup \{m\}))\in\CrSeq(K).
	\]		
		Thus, $\lambda\in\CrCell(K)$, since it is the critical cell of the sequence above.
	}
	\item{
		\textit{$\lambda=\pi|m|B|\varrho$ for $(C,B,D)\in\Br(K)$, $\pi\in \Crit(\cW_{K|^0_{C}})$, $\varrho\in\Crit(\cW_{K|^1_{D}})$}. By the inductive hypothesis,
\begin{align*}
	\pi&=\sigma((E^\pi_j)_{j=1}^q,(F^\pi_j)_{j=0}^q),\\
	\varrho&=\sigma((E^\varrho_j)_{j=1}^r,(F^\varrho_j)_{j=0}^r),
\end{align*}		
		for $((E^\pi_j),(F^\pi_j))\in\CrSeq(K|^0_C)$, $((E^\varrho_j),(F^\varrho_j))\in\CrSeq(K|^1_D)$. Now $\lambda$ is associated to a sequence
		\[
			((E_1^\pi,\dots,E_q^\pi,B\cup\{m\},E_1^\varrho,\dots,E_r^\varrho),(F_0^\pi,\dots,F_q^\pi,F_0^\varrho,\dots,F^\varrho_r)),
		\]
		which is critical in $K$; the only non-trivial fact to check is that either $E^\pi_q=\emptyset$ or $\max(E^\pi_q)<\max(B\cup\{m\})$, which is guaranteed since $m$ is a maximal element of $A$.
	}
	\end{itemize}
	We have shown that $\Crit(\cW_k)\cong\CrSeq(K)$ and it is clear that dimension is preserved.
\end{proof}

\begin{thm}\label{t:Main}
	Let $K$ be an $A$--cubical complex. Then $|\cP_K|\simeq \vP(|K|)_\bO^\bI$ is homotopy equivalent to a CW-complex $X_K$ that has exactly $\# \CrSeq^d(K)$ cells of dimension $d$.
\end{thm}
\begin{proof}
	By \ref{p:WKIsGradient}, $\cW_K$ is a gradient field and by \ref{p:ExplicitDescOfCrit} the number of critical cells of dimension $d$ equals to $\# \CrSeq^d(K)$. The conclusion follows from \ref{t:Forman}.
\end{proof}

\section{Euclidean cubical complexes}

Euclidean cubical complexes \cite{RZ} constitute a class of semi-cubical sets which is especially important for applications in concurrency, since they include state spaces of PV-programs \cite{D,Z-PV} . We recall the definition here and show that Euclidean cubical complexes can be regarded as $A$--cubical complexes.

\begin{df}
	\emph{An elementary cube} in $\vR^n$ is a subset having the form
	\[
		[\ba,\bb]=\{\bx\in\vR^n:\; \ba\leq\bx\leq\bb\}=\{(x_1,\dots,x_n)\in \vR^n:\; \forall_{i=1}^n\; a_i\leq x_i\leq b_i\},
	\]
	where $\ba=(a_1,\dots,a_n),\bb=(b_1,\dots,b_n)\in \Z^n$ and $b_i-a_i\in\{0,1\}$ for all $i$. The integer $|\bb-\ba|=\sum_{i=1}^n (b_i-a_i)$ is \emph{the dimension} of the cube $[\bb,\bl]$. A set
	\[
		\dir([\ba,\bb])=\{i\in\{1,\dots,n\}:\; b_i-a_i=1\}
	\]
	is \emph{the set of directions} of the cube $[\ba,\bb]$.
	
	\emph{A Euclidean cubical complex} $K$ in $\vR^n$ is a family of elementary cubes in $\vR^n$ that is closed with respect to taking subsets. $K$ can be regarded as a semi-cubical set: $K[d]$ is the set of all $d$--dimensional elementary cubes of $K$, and the face maps are defined as follows. If $[\ba,\bb]\in K[d]$  and  $\dir([\ba,\bb])=\{r(1)<\dots<r(d)\}$, then
	\[
		d^\varepsilon_i([\ba,\bb])=[\ba+\varepsilon \be_{r(i)}, \bb-(1-\varepsilon)\be_{r(i)} ],
	\]
	where $\be_i=(0,\dots,0,1,0,\dots,0)$ and $1$ stands at the $i$--th place.
	This corresponds to taking the lower or the upper face in the $i$--th direction of the cube.
\end{df}

\begin{rem}
	The geometric realization of a Euclidean complex $K$ regarded as a semi-cubical set is homeomorphic, in a canonical way, with a sum of this cubes regarded as subsets of $\R^n$.
\end{rem}

%State spaces of PV-programs are geometric realizations of Euclidean cubical complexes; this leads to applications in concurrency.

Let $K$ be a finite Euclidean cubical complex. Without loss of generality, we can assume that $K$ is contained in a hyperrectangle $[\bO,\bk]$, $\bk=(k_1,\dots,k_n)\in\Z^n$, since $K$ can be shifted if necessary. Define a poset
\begin{equation}\label{e:Ak}
	A_\bk=\{(1,1)<(1,2)<\dots<(1,{k_1})<(2,1)<\dots<(2,{k_2})<\dots<(n,1)<\dots<(n,{k_n})\}
\end{equation}
For an elementary cube $[\ba,\bb]\subseteq [\bO,\bk]$ having dimension $d$, define $i_\bk([\ba,\bb])\in \square^{A_\bk}[d]$ by
\begin{equation}\label{e:EuclInc}
	i_\bk([\ba,\bb])(i,j)=\begin{cases}
		0 & \text{for $b_i< j$,}\\
		1 & \text{for $j\leq a_i$,}\\
		* & \text{for $a_i<j=b_i$.}
	\end{cases}
\end{equation}
This definition is valid since $b_i-a_i\in\{0,1\}$ for all $i$, and defines an injective semi-cubical map $[\bO,\bk]\to \square^{A_\bk}$; it is elementary to check that this commutes with the face maps. Denote
\begin{equation}
	\boxplus^\bk:=i_\bk([\bO,\bk])\subseteq \square^{A_\bk}.
\end{equation}
For an Euclidean cubical subcomplex $K\subseteq [\bO,\bk]$ obviously $i_\bk(K)\subseteq \boxplus^\bk \subseteq \square^{A_\bk}$. Thus, $K$  can be regarded as an $A_\bk$--cubical complex.

\begin{rem}
	Instead of the order $(\ref{e:Ak})$ we can use any order such that $(i,j)<(i,j')$ for $j<j'$. This leads to a different vector field $\cW_{i_\bk(K)}$, with possibly another set of critical cells.
\end{rem}

The following observation is elementary but will be frequently used
\begin{prp}\label{p:BoxDotCriterion}
	Let $\lambda=B_1|\dots|B_l\in \cP_{A_\bk}$. The following conditions are equivalent:
	\begin{enumerate}[\normalfont (a)]
	\item{$\lambda\in \cP_{{\boxplus}^\bk}$.}
	\item{For every $i\in\{1,\dots,n\}$ and $j<j'\in \{1,\dots,k_i\}$, if $(i,j)\in B_r$ and $(i,j')\in B_{r'}$, then $r<r'$.\qed}
	\end{enumerate}
\end{prp}

For a subset ${B}\subseteq A_\bk$, let $\bar B$ be a multiset that contains only elements from $\{1,\dots,n\}$, and every $i\in\{1,\dots,n\}$ is contained in $\bar{B}$ with multiplicity $\#\{j\in \{1,\dots,k_i\}:\; (i,j)\in {E} \}$. $\bar{B}$ is the image of ${B}$ under the projection $A_\bk\ni (i,j)\mapsto i\in  \{1,\dots,n\}$, with multiplicities preserved. In terms of characteristic functions, we have
\begin{equation}
	\chi_{\bar{B}}(i)=\sum_{j=1}^{k_i} \chi_{{B}}(i,j).
\end{equation}

\begin{prp}\label{p:barBIsaSet}
	If $\lambda=B_1|\dots|B_l\in \cP_{{\boxplus}^\bk}$, then, for every $r\in\{1,\dots,l\}$, $\bar{B}_r$ is a set. 
\end{prp}
\begin{proof}
	This follows from Proposition \ref{p:BoxDotCriterion}.
\end{proof}

Let $[\bk]$ denote the multiset having characteristic function $\bk$, i.e., such that contains $i\in\{1,\dots,n\}$ exactly $k_i$ times.
\emph{An ordered partition of $[\bk]$} is a sequence $\mu=C_1|C_2|\dots|C_l$, where $C_i$ are multisets with all elements in $\{1,\dots,n\}$, such that $\sum_{i=1}^l \chi_{C_i}=\bk$. An ordered partition $\mu$ is \emph{proper} if all multisets $C_i$ are sets, i.e., $\chi_{C_i}\leq \bI$. Let $\cR_\bk$ be the poset of ordered partitions of $\bk$, ordered by refinement, and let $\cR_\bk^{pr}\subseteq \cR_\bk$ be the subposet of proper partitions.

\begin{prp}
	If $\lambda=B_1|\dots|B_l\in \cP_{{\boxplus}^\bk}$, then $\bar\lambda=\bar{B}_1|\dots|\bar{B}_l\in \cR_\bk^{pr}$.
\end{prp}
\begin{proof}
	For $i\in\{1,\dots,n\}$ we have
	\[
		\sum_{r=1}^l \chi_{\bar{B}_r} (i)= \sum_{r=1}^l \sum_{j=1}^{k_i} \chi_{B_r}(i,j)= \sum_{j=1}^{k_i}  \chi_{\bigcup_{r=1}^l B_r}(i,j)=\sum_{j=1}^{k_i}  \chi_{A_\bk}(i,j)=k_i.
	\]
	Thus, $\bar\lambda$ is an ordered partition of $[\bk]$, and by \ref{p:barBIsaSet} this is proper.
\end{proof}

\begin{prp}
	For every proper ordered partition $E_1|E_2|\dots|E_l$ of $[\bk]$ there exists a unique ordered partition $\lambda=B_1|\dots|B_l\in \cP_{{\boxplus}^\bk}$ such that $E_r=\bar{B}_r$ for all $r\in\{1,\dots,l\}$.
\end{prp}
\begin{proof}
	Define
	\begin{equation}
		B_r=\{(i,j):\; \text{$i\in E_r$ and $j=1+\#\{s\in\{1,\dots,r-1\}:\; i\in E_s\}$} \}.
	\end{equation}
	Clearly, $\bar{B}_r=E_r$ and $\lambda=B_1|\dots|B_l$ is a partition of $A_\bk$, which satisfies the condition \ref{p:BoxDotCriterion}.(b). Thus, $\lambda\in \cP_{{\boxplus}^\bk}$. On the other hand, if $i\in E_r$, then $B_r$ must contain a pair $(i,j)$ and \ref{p:BoxDotCriterion}.(b) enforces that $j=1+\#\{s\in\{1,\dots,r-1\}\}$. 
\end{proof}

As a consequence, the formula
\begin{equation}
	U_\bk: \cP_{{\boxplus}^\bk}\ni \lambda=B_1|B_2|\dots|B_l\mapsto \bar{\lambda}=\bar{B}_1|\bar{B}_2|\dots|\bar{B}_l \in \cR_\bk.
\end{equation}
defines an isomorphism of posets. For a Euclidean cubical complex $K\subseteq [\bO,\bk]$ let 
\begin{equation}
	\cR_K=U_\bk(\cP_{i_\bk(K)}).
\end{equation}

\begin{prp}
	For a Euclidean cubical complex $K\subseteq [\bO,\bk]$, we have
	\[
		\cR_K:=\{E_1|\dots|E_l\in\cR^{pr}_{\bk}:\; \forall_{r\in\{1,\dots,l\}}\; [\sum_{s=1}^{r-1} \chi_{E_s},\sum_{s=1}^{r} \chi_{E_s}]\in K\}.
	\]
\end{prp}
\begin{proof}
	The definition of $i_\bk$ implies that
	\begin{equation}
		c(B_1\cup\dots\cup B_{r-1},B_r,B_{r+1}\cup\dots \cup B_l)=i_\bk([\sum_{s=1}^{r-1} \chi_{\bar{B}_s},\sum_{s=1}^{r} \chi_{\bar{B}_s}]).
	\end{equation}
	The conclusion follows.
\end{proof}

\begin{prp}\label{p:barEIsaSet}
	Assume that $(({E}_j)_{j=1}^q,({F}_j)_{j=0}^q)\in\CrSeq(i_\bk(K))$. Then, for all $j$, $\bar{E}_j$ is a set.
\end{prp}	
\begin{proof}
	The associated critical cell $\sigma(({E}_j),({F}_j))$ belongs to $\cP_{i_\bk(K)}$ and has the form $\lambda|(r,s)|E'_j|\mu$, where $(r,s)=\max(E_j)$, $E'_j=E_j\setminus(r,s)$. Then, by \ref{p:barBIsaSet}, $\bar{E}'_j$ is a set. If $\max(\bar{E}'_j)=r$, then $(r,s')\in E'_j$ for some $s'$. But \ref{p:BoxDotCriterion} implies that $s<s'$, which contradicts the assumption that $(r,s)=\max(E_j)$. Thus, $r>\max(\bar{E}'_j)$ and then $\bar{E}_j=\{r\}\cup \bar{E}'_j$ is a set.
\end{proof}

For $\ba\leq \bb\in\Z^n$, \emph{the minimal line from $\ba$ to $\bb$} is a Euclidean cubical complex $L_{\ba,\bb}$ such that
\begin{equation}
	|L_{\ba,\bb}|=\bigcup_{i=1}^n \{a_1\}\times \dots \times \{a_{i-1}\} \times [a_i,b_i] \times \{b_{i+1}\} \times\dots\times \{b_n\}\subseteq \R^n.
\end{equation}  

\begin{df}\label{d:Route}
	\emph{A route to $\bk$}, where $0\leq \bk\in \Z^n$, is a pair of sequences  $((\ba^j)_{j=1}^{q+1}, (\bb^j)_{j=0}^{q})$ of points of $\Z^n$ such that
	\begin{enumerate}[\normalfont (a)]
		\item{$\bb^0=\bO$}
		\item{$\ba^{q+1}=\bk$,}
		\item{$\bb^j\leq \ba^{j+1}$ for $0\leq j\leq q$,}
		\item{$\bO<\bb^j - \ba^j \leq\bI$ for $0<j\leq q$.}			
	\end{enumerate}
	Let $K\subseteq [\bO,\bk]$ be a Euclidean complex. \emph{A critical route in $K$} is a route to $\bk$ such that
	\begin{enumerate}[\normalfont (a)]
		\addtocounter{enumi}{4}
		\item{$L_{\bb^{j},\ba^{j+1}}\subseteq K$ for $0\leq j\leq q$,}
		\item{$[\ba^j,\bb^j]\not\in K$ for $0<j\leq q$,}
		\item{$[\ba^j,\ba^j+\be_{m_j}], [\ba^j+\be_{m_j},\bb_j]\in K$ for $0<j\leq q$, where $m_j=\max(\dir([\ba^j,\bb^j]))$.}
	\end{enumerate}	
\end{df}

The following picture illustrates all critical routes in an exemplary Euclidean cubical set.
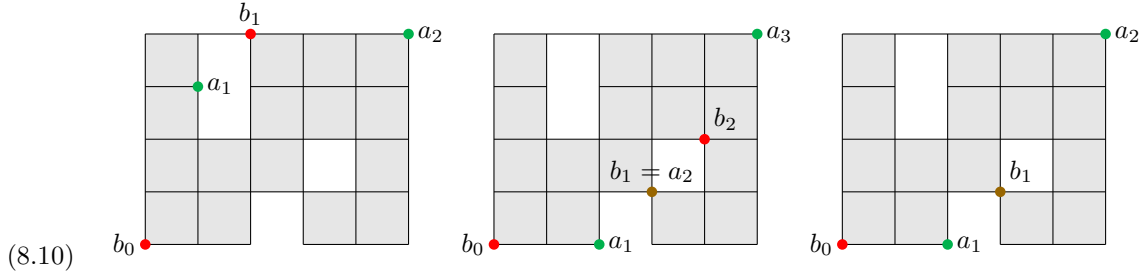
\begin{equation}
\begin{tikzpicture}[scale=0.7]
	\fill[color=gray!20!white] 	(0,0) -- (5,0)--(5,4)--(0,4)--(0,0);
	\fill[color=white] 	(2,0) -- (3,0)--(3,1)--(2,1)--(2,0);
	\fill[color=white] 	(3,1) -- (4,1)--(4,2)--(3,2)--(3,1);
	\fill[color=white] 	(1,2) -- (2,2)--(2,4)--(1,4)--(1,2);

	\draw (0,0)--(2,0);
	\draw (3,0)--(5,0);
	\draw (0,1)--(5,1);
	\draw (0,2)--(5,2);
	\draw (0,3)--(1,3); \draw(2,3)--(5,3);
	\draw (0,4)--(5,4);

	\draw(0,0)--(0,4);	
	\draw(1,0)--(1,4);	
	\draw(2,0)--(2,4);	
	\draw(3,0)--(3,4);	
	\draw(4,0)--(4,4);	
	\draw(5,0)--(5,4);	

	\fill[color=red] (0,0) circle[radius=0.1];
	\node[left] at (0,0) {$b_0$};

	\fill[color=blue!30!green] (1,3) circle[radius=0.1];
	\node[right] at (1,3) {$a_1$};

	\fill[color=red] (2,4) circle[radius=0.1];
	\node[above] at (2,4) {$b_1$};

	\fill[color=blue!30!green] (5,4) circle[radius=0.1];
	\node[right] at (5,4) {$a_2$};
\end{tikzpicture}
\begin{tikzpicture}[scale=0.7]
	\fill[color=gray!20!white] 	(0,0) -- (5,0)--(5,4)--(0,4)--(0,0);
	\fill[color=white] 	(2,0) -- (3,0)--(3,1)--(2,1)--(2,0);
	\fill[color=white] 	(3,1) -- (4,1)--(4,2)--(3,2)--(3,1);
	\fill[color=white] 	(1,2) -- (2,2)--(2,4)--(1,4)--(1,2);

	\draw (0,0)--(2,0);
	\draw (3,0)--(5,0);
	\draw (0,1)--(5,1);
	\draw (0,2)--(5,2);
	\draw (0,3)--(1,3); \draw(2,3)--(5,3);
	\draw (0,4)--(5,4);

	\draw(0,0)--(0,4);	
	\draw(1,0)--(1,4);	
	\draw(2,0)--(2,4);	
	\draw(3,0)--(3,4);	
	\draw(4,0)--(4,4);	
	\draw(5,0)--(5,4);	

	\fill[color=red] (0,0) circle[radius=0.1];
	\node[left] at (0,0) {$b_0$};

	\fill[color=blue!30!green] (2,0) circle[radius=0.1];
	\node[right] at (2,0) {$a_1$};

	\fill[color=red!60!green] (3,1) circle[radius=0.1];
	\node[above] at (3,1) {$b_1=a_2$};

	\fill[color=red] (4,2) circle[radius=0.1];
	\node[above right] at (4,2) {$b_2$};

	\fill[color=blue!30!green] (5,4) circle[radius=0.1];
	\node[right] at (5,4) {$a_3$};
\end{tikzpicture}
\begin{tikzpicture}[scale=0.7]
	\fill[color=gray!20!white] 	(0,0) -- (5,0)--(5,4)--(0,4)--(0,0);
	\fill[color=white] 	(2,0) -- (3,0)--(3,1)--(2,1)--(2,0);
	\fill[color=white] 	(3,1) -- (4,1)--(4,2)--(3,2)--(3,1);
	\fill[color=white] 	(1,2) -- (2,2)--(2,4)--(1,4)--(1,2);

	\draw (0,0)--(2,0);
	\draw (3,0)--(5,0);
	\draw (0,1)--(5,1);
	\draw (0,2)--(5,2);
	\draw (0,3)--(1,3); \draw(2,3)--(5,3);
	\draw (0,4)--(5,4);

	\draw(0,0)--(0,4);	
	\draw(1,0)--(1,4);	
	\draw(2,0)--(2,4);	
	\draw(3,0)--(3,4);	
	\draw(4,0)--(4,4);	
	\draw(5,0)--(5,4);	

	\fill[color=red] (0,0) circle[radius=0.1];
	\node[left] at (0,0) {$b_0$};

	\fill[color=blue!30!green] (2,0) circle[radius=0.1];
	\node[right] at (2,0) {$a_1$};

	\fill[color=red!60!green] (3,1) circle[radius=0.1];
	\node[above right] at (3,1) {$b_1$};

	\fill[color=blue!30!green] (5,4) circle[radius=0.1];
	\node[right] at (5,4) {$a_2$};
\end{tikzpicture}
\end{equation}

We will prove that there is 1--1 correspondence between critical sequences in $i_\bk(K)$ and critical routes in $K$.

 For $(({E}_j)_{j=1}^q,({F}_j)_{j=0}^q)\in\CrSeq(i_\bk(K))$ define 
\begin{align}
	\ba^j&=\chi_{\bar F_0}+\chi_{\bar E_1}+\chi_{\bar F_1}+\dots+\chi_{\bar E_{j-1}}+\chi_{\bar F_{j-1}}\\
	\bb^j&=\chi_{\bar F_0}+\chi_{\bar E_1}+\chi_{\bar F_1}+\dots+\chi_{\bar E_{j-1}}+\chi_{\bar F_{j-1}}+\chi_{\bar E_j}. \notag
\end{align}
We will check that $((\ba^j)_{j=1}^{q+1}, (\bb^j)_{j=0}^{q})$ satisfies the conditions (a)-(g) of Definition \ref{d:Route}. Points (a) and (c) and obvious, and (b) follows from \ref{d:CrSeq}.(a), since $\chi_{\bar{A}_\bk}=\bk$. Points (e) and (g) follow from \ref{d:CrSeq}.(b) and the definitions of $\tau_{F_j}$ and $\kappa_{E_j}$ (\ref{e:TauKappa}), respectively. Finally, point (d) follows from \ref{p:barEIsaSet} and (f) follows from \ref{d:CrSeq}.(d). Thus, $((\ba^j),(\bb^j))$ is a critical route in $K$.

On the other hand, if $((\ba^j)_{j=1}^{q+1}, (\bb^j)_{j=0}^{q})$ is a critical route in $K$, then we define sequences of subsets of $A_\bk$: 
\begin{align}
	E_j &= \{(i, b^j_i):\; i\in \dir([\ba^j,\bb^j])\}=\{(i,r):\; i\in\{1,\dots,n\},\; a^{j}_i < r \leq b^j_i\}\\
	F_j &= \{(i,r):\; i\in\{1,\dots,n\},\; b^{j-1}_i < r \leq a^j_i\}.	
\end{align}
The argument similar as above shows that $((E_j)_{j=1}^q, (F_j)_{j=0}^q)$ is a critical sequence in an $A_\bk$--cubical complex $i_\bk(K)$.

Both constructions preserve the dimension. Let $\Rt^d(K)$ be the set of critical routes in $K$ having dimension $d$. By combining the argument above with Proposition \ref{p:ExplicitDescOfCrit}, we obtain

\begin{prp}\label{p:ExplicitEu}
	Let $K\subseteq [\bO,\bk]$ be a Euclidean cubical complex. For every $d\geq 0$ there are bijections
	\[
		\Crit^d_{\cP_{i_\bk(K)}}(\cW_{i_\bk(K)}) \simeq \CrSeq^d(i_\bk(K)) \simeq \Rt^d(K).\qed
	\]
\end{prp}

Here follows the main theorem of this Section.

\begin{thm}\label{t:MainEuclidean}
	Let $\bk\in \Z^n_{\geq 0}$ and let $K\subseteq [\bO,\bk]$ be a Euclidean cubical complex. Then $\vP(|K|)_\bO^\bk$ is homotopy equivalent to a CW-complex $X_K$ which has exactly $\#\Rt^d(K)$ cells of dimension $d$.
\end{thm}
\begin{proof}
	This follows from Theorem \ref{t:Main} and Proposition \ref{p:ExplicitEu}.
\end{proof}

\section{Applications}

At the first glance, it is not clear how efficient is the description of the space of directed paths on an $A$--cubical complex provided in Theorems \ref{t:Main} and \ref{t:MainEuclidean}. In this Section we describe three cases in which this description is optimal, i.e., the cells of the CW-complex $X_K$ correspond to the generators of the homology of $\vP(|K|)_\bO^\bI$.

\subsection*{"Not $(s+1)$--equal configuration spaces"}
For integers $0<s<n$, the "not $(s+1)$-equal" configuration space is 
\begin{equation}
		\Conf_{n,s}(\R)=\{(t_i)_{i=1}^n\in \R^n:\; \forall_{t\in\R}\; \#\{i:\; t_i=t\}\leq s\}.
\end{equation}
Homology groups of these spaces were calculated by Bj{\"o}rner and Welker \cite{BW}. We will reprove their result here. 

\emph{The $q$-skeleton} of a semi-cubical complex $K$ is a cubical complex $K_{(q)}$ such that
\begin{equation}
	K_{(q)}[d]=\begin{cases}
		K[d] & \text{for $d\leq q$}\\
		\emptyset & \text{for $d>q$.}
	\end{cases}
\end{equation}
and the face maps of $K_{(q)}$ are inherited from $K$. 

Fix $n>0$ and let $A=\{1<2<\dots<n\}$; we will write $\square^n$ instead of $\square^{\{1<2<\dots<n\}}$.

\begin{prp}\label{p:ListSpac}
	For $0<s\leq n$, the following spaces are homotopy equivalent:
	\begin{enumerate}[\normalfont (a)]
		\item{$\vP(|\square^n_{(s)}|)_\bO^\bI$,}
		\item{$|\cP_{\square^n_{(s)}}|$,}
		\item{$\Conf_{n,s}(\R)$.	}
	\end{enumerate}
\end{prp}
\begin{proof}
	This is a consequence of Section 2 and  \cite[Section 2.2]{MR}.
\end{proof}

The space $\vP(|\square^n_{(s)}|)_\bO^\bI$ plays an important role in concurrency, since it is the execution space of a PV-program consisting of $n$ processes and each of then uses once a resource having capacity $s$.

\begin{prp}\label{p:CrSeqNotSEqual}
	Fix $s>1$. Then $\CrSeq^d(\square^n_{(s)})=\emptyset$ if $(s-1)$ does not divide $d$; otherwise,
	\[
		\CrSeq^{q(s-1)}(\square^n_{(s)})=\{((E_j)_{j=1}^{q},(F_j)_{j=0}^q):\; \forall_j\; \#E_j=s+1 \text{ and } \max(F_{j-1})<\max(E_j)\},
	\]
	where $E_1\cupdot \dots \cupdot E_q\cupdot F_0\cupdot \dots \cupdot F_q=\{1,\dots,n\}$.
\end{prp}
\begin{proof}
	The condition (b) in Definition \ref{d:CrSeq} implies that $\# E_j\leq s+1$, and the condition (d) implies that $\#E_j>s$ for all $j$.
\end{proof}

As a consequence, we obtain

\begin{prp}\label{p:CnsCritical}
	For $0<s\leq n$, $\vP(\square^n_{(s)})_\bO^\bI$ is homotopy equivalent to a CW-complex that has exactly $\# \CrSeq^d(\square^n_{(s)})$ cells of dimension $d$. If $s>2$, then
	\[
		H_d(\vP(\square^n_{(s)})_\bO^\bI)=\begin{cases}
			\Z^{b(n,s,q)} & \text{for $d=q(s-1)$,}\\
			0 & \text{otherwise,}
		\end{cases}
	\]
	where $b(n,s,q):=\#\Crit^{q(s-1)}(\square^n_{(s)})$.
\end{prp}
\begin{proof}
	The first statement follows from \ref{t:Main} and \ref{p:CrSeqNotSEqual}. If $s>1$, no cells having consecutive dimensions appear, which implies the second statement.
\end{proof}

For $s=2$, a calculation of the homology groups requires checking that the incidence numbers of cells having consecutive dimensions always vanish. This can be done using methods from, for example, \cite[Chapter 11]{Kozlov}. We omit these technical calculations here.

\subsection*{Generalized not $(s+1)$--equal configuration spaces.} 

For $\bk=(k_1,\dots,k_n)\in \Z_+^n$ define a generalized "not $(s+1)$--equal" configuration space:
\begin{equation}
	\Conf_{\bk,s}(\R)=\{(t_i^j)_{i=1,\dots,n}^{j=1,\dots,k_i}:\; \forall_i\; t_i^1<\dots<t_i^{k_i} \text{ and } \forall_{t\in\R}\; \#\{(i,j):\; t^j_i=t\}\leq s\}.
\end{equation}
 This is a generalization of "not $(s+1)$--equal" configuration spaces, since $\Conf_{n,s}(\R)=\Conf_{(1,\dots,1),s}(\R)$.

The following proposition is an analogue of \ref{p:ListSpac}, and its proof is similar.
\begin{prp}
	For $\bk\in\Z_+^n$, the following spaces are homotopy equivalent:
	\begin{enumerate}[\normalfont (a)]
		\item{$\vP([\bO,\bk]_{(s)})_\bO^\bk$,}
		\item{$\vP(|{\boxplus^\bk_{(s)}}|)_\bO^\bI$,}
		\item{$|\cP_{{\boxplus}^\bk_{(s)}}|$,}
		\item{$\Conf_{\bk,s}(\R)$.\qed}
	\end{enumerate}
\end{prp} 

In terms of PV-programs, the space $\vP([\bO,\bk]_{(s)})_\bO^\bk$ is the execution space of a PV-program with a single resource of capacity $s$ and $n$ processes. The $i$--th process acquires the resource exactly $k_i$ times.

\begin{prp}
	Fix $s>0$. A route $((\ba^j)_{j=1}^{q+1}, (\bb^j)_{j=0}^{q})$ to $\bk$ is a critical route in $[\bO,\bk]_{(s)}$ if and only if $\dim([\ba^j,\bb^j])=s+1$  for $j\in\{1,\dots,q\}$. 
	In particular, the dimension of the critical route equals $q(s-1)$.
\end{prp}
\begin{proof}
	The condition \ref{d:Route}.(e) is satisfied since $s>0$, and the condition $\dim([\ba^j,\bb^j])=s+1$ is equivalent to the conditions \ref{d:Route}.(f)-(g).
\end{proof}

Immediately from Theorem \ref{t:MainEuclidean} follows the analogue of Proposition \ref{p:CnsCritical}.

\begin{prp}
	Let $b_\bk(n,s,q)$ be the number of $q(s-1)$--dimensional routes in $[\bO,\bk]_{(s)}$. Then $\vP([\bO,\bk]_{(s)})_\bO^\bk$ is homotopy equivalent to a CW-complex that has exactly $b_\bk(n,s,q)$ cells of dimension $q(s-1)$ and no cells having dimension non-divisible by $(s-1)$. As a consequence,
	\[
		H_d(\vP(\square^n_{(s)})_\bO^\bI)=\begin{cases}
			\Z^{b_\bk(n,s,q)} & \text{for $d=q(s-1)$,}\\
			0 & \text{otherwise,}
		\end{cases}
	\]
	for $s>2$.
\end{prp}

This recovers results obtained by Meshulam and Raussen in \cite[Section 5.3]{MR}.

\subsection*{Directed path spaces on $K$ for $[\bO,\bk]_{(n-1)}\subseteq K\subseteq [\bO,\bk]$}
This case was considered in \cite{RZ}.

\begin{prp}
	Assume that $n\geq 2$, $\bO\leq \bk\in \Z^n$ and that $K$ is a Euclidean cubical complex such that $[\bO,\bk]_{(n-1)}\subseteq K\subseteq [\bO,\bk]$. Then a route to $\bk$ $((\ba^j)_{j=1}^{q+1}, (\bb^j)_{j=0}^{q})$ is a critical route in $K$ if and only if $[\ba^j,\bb^j]\not\in K$ for $j\in\{1,\dots,n\}$. In particular, this implies that $\ba^j=\bb^j-\bI$ for $j\in\{1,\dots,n\}$ and that the dimension of the critical route equals $q(n-2)$.
\end{prp}
\begin{proof}
	The conditions (e) and (g) in Definition \ref{d:Route} are trivially satisfied so only the condition (f) remains.
\end{proof}

Note that there is 1-1 correspondence between critical routes in $K$ and cube sequences in $K$ defined in \cite[Section 1.4]{RZ}: if $((\ba^j)_{j=1}^{q+1}, (\bb^j)_{j=0}^{q})$ is a critical route in $K$, then $[\bb^1,\dots,\bb^q]$ is a cube sequence and, inversely, a cube sequence $[\bb^1,\dots,\bb^q]$ determines a critical route $((\bb^j-\bI)_{j=1}^{q+1}, (\bb^j)_{j=0}^{q})$ (where $\bb^0=\bO$, $\bb^{q+1}=\bk+\bI$).

Thus, the main theorem of \cite{RZ} (Theorem 1.1) follows immediately from Theorem \ref{t:MainEuclidean} if $n\neq 3$ since there are no critical routes having consecutive dimensions. For $n=3$, the homology calculation requires, as in the previous cases, some additional calculations we do not present here.

\end{document}